\documentclass[10pt]{amsart}
\usepackage{amsmath,amscd}
\usepackage{amsbsy}
\usepackage{amssymb}
\usepackage{amscd,amsthm}
\usepackage[all,cmtip]{xy}

\newtheorem{thm}{Theorem}\numberwithin{thm}{section}
\newtheorem{lem}[thm]{Lemma}
\newtheorem{prop}[thm]{Proposition}
\newtheorem{cor}[thm]{Corollary}
\newtheorem{exam}[thm]{Example}
\newtheorem{rema}[thm]{Remark}
\newtheorem{prob}[thm]{Problem}

\newtheorem{defi}[thm]{Definition}

\newtheorem*{thm2}{Theorem}

\newtheorem*{con2}{Conjecture}

\begin{document}
\begin{center}
\huge{Rational maps between varieties associated to central simple algebras}\\[1cm]
\end{center}
\begin{center}

\large{Sa$\mathrm{\check{s}}$a Novakovi$\mathrm{\acute{c}}$}\\[0,5cm]
{\small February 2016}\\[0,5cm]
\end{center}
{\small \textbf{Abstract}. 
In this paper we show that if two central simple $k$-algebras generate the same cyclic subgroup in $\mathrm{Br}(k)$, then there are rational maps between varieties associated to these algebras, such as Brauer--Severi varieties, norm hypersurfaces and symmetric powers. In some cases we even have rational embeddings. We also relate the obtained results to the Amitsur conjecture.
\begin{center}
\tableofcontents
\end{center}
\section{Introduction}
To central simple $k$-algebras one can associate several algebraic varieties. Brauer--Severi varieties are maybe the most prominent one. It is well-known that central simple algebras are in one-to-one correspondence with Brauer--Severi varieties over a field $k$ \cite{AR}. So it is natural to study the geometry of a Brauer--Severi variety in dependence on the algebraic structure of the corresponding algebra and vice versa. Amitsur \cite{AM} investigated so called generic splitting fields and proved that if two Brauer--Severi varieties $X$ and $Y$ are birational then the corresponding central simple algebras $A$ and $B$ generate the same cyclic subgroup in $\mathrm{Br}(k)$. So it was natural to ask if the other implication is true as well. And indeed, Amitsur proved that this is true for certain ground fields $k$. Consequently he conjectured that two Brauer--Severi varieties $X$ and $Y$ are birational if and only if the corresponding central simple algebras $A$ and $B$ generate the same cyclic subgroup in $\mathrm{Br}(k)$. This conjecture is nowadays called the Amitsur conjecture for Brauer--Severi varieties. Several results in favor of this conjecture are known (see \cite{AM}, \cite{KR}, \cite{KR1}, \cite{RO}, \cite{TR}). In view of this conjecture it is an important task to understand more closely the relation between the geometry of Brauer--Severi varieties and the structure of the corresponding central simple algebras in $\mathrm{Br}(k)$.

Let $D$ and $D'$ by division algebras of the same degree and $X$ and $Y$ the corresponding Brauer--Severi varieties. Furthermore, denote by $\mathcal{X}$ and $\mathcal{Y}$ the Brauer--Severi varieties corresponding to $M_n(D)$ and $M_n(D')$ for an integer $n>1$. Note that there are closed immersions $X\hookrightarrow \mathcal{X}$ and $Y\hookrightarrow \mathcal{Y}$. 

Now if $D$ and $D'$ generate the same subgroup in $\mathrm{Br}(k)$ one gets dominant rational maps $X\dashrightarrow Y$ and $Y\dashrightarrow X$ (see \cite{KO}). But in general, the closed immersions $X\hookrightarrow \mathcal{X}$ and $Y\hookrightarrow \mathcal{Y}$ do not induce rational embeddings for instance of $X$ into $\mathcal{Y}$. Moreover, the rational map $X\dashrightarrow Y$ is also far from being birational. In this context, our first main result will be the following.
\begin{thm2}(Theorem 5.3)
Let $\mathcal{X}$ and $\mathcal{Y}$ be two Brauer--Severi varieties corresponding to the central simple $k$-algebras $A=M_n(D)$ and $B=M_n(D')$ with $n>1$ arbitrary and assume $\mathrm{deg}(D)=\mathrm{deg}(D')$. Denote by $X$ and $Y$ the Brauer--Severi varieties corresponding to $D$ and $D'$. Then $A$ and $B$ generate the same cyclic subgroup in $\mathrm{Br}(k)$ if and only if there are rational embeddings $X\dashrightarrow \mathcal{Y}$ and $Y\dashrightarrow \mathcal{X}$.
\end{thm2}
By definition, the rational embedding $X\dashrightarrow \mathcal{Y}$ from above can be factored as $U\rightarrow Z\rightarrow \mathcal{Y}$, where the first arrow is an open and the second one a closed immersion. Here $U$ is an suitable open subset of $X$. Clearly, if $Z=Y$ then $X$ and $Y$ are birational.
We briefly discuss the case $Z\neq Y$ and consider the closed subscheme $Z\cap Y\subset Y$ (we show that $Z\cap Y$ can always assumed to be non-empty). In this case we relate the irreducible components of $Z\cap Y$ to the so called $AS$-bundles of $Y$ via the maps between the involved Grothendieck and Chow groups (Theorem 5.11). For details on $AS$-bundles on arbitrary proper $k$-schemes, and Brauer--Severi varieties in particular, we refer to \cite{N} where these bundles are introduced.

It is also possible to associate with a central simple algebra $A=M_n(D)$ the so called norm hypersurface $V(A)$ (see Section 6). Studying the function fields of this norm hypersurfaces, Saltman \cite{SA} proved a variant of the Amitsur conjecture, namely, that two central simple $k$-algebras $A=M_n(D)$ and $B=M_n(D')$ of the same degree generate the same cyclic subgroup in $\mathrm{Br}(k)$ if and only if $V(A)$ is birational to $V(B)$. In \cite{SA1} Saltman constructed so called rational embeddings (see Section 5) of the Brauer--Severi variety $\mathcal{X}$ corresponding to $A$ into the norm hypersurface $V(A)$. An overall strategy to tackle the Amitsur conjecture is the following: Suppose $A$ and $B$ generate the same cyclic subgroup so that $V(A)$ and $V(B)$ are birational. Now construct rational embeddings $\mathcal{X}\dashrightarrow V(A)$ and $\mathcal{Y}\dashrightarrow V(B)$ such that the birationality of the norm hypersurfaces induces a birational map between $\mathcal{X}$ and $\mathcal{Y}$. This idea was captured by Meth \cite{ME} in his thesis, where the set of rational embeddings is enlarged and some ideas are refined. Using Theorem 5.3 from above, in Section 6 of the present paper we prove the following theorem. 
\begin{thm2}(Theorem 6.4)
Let $A$ and $B$ central simple algebras of the same degree and $\mathcal{X}$ and $\mathcal{Y}$ the corresponding Brauer--Severi varieties. Denote by $X$ and $Y$ the minimal linear subvarieties. Then $A$ and $B$ generate the same cyclic subgroup in $\mathrm{Br}(k)$ if and only if there are rational embeddings $X\dashrightarrow V(B)$ and $Y\dashrightarrow V(A)$ such that the image of their domains lies in the smooth locus $V^+_B$ respectively $V^+_A$.
\end{thm2}
Note that it is also possible to get rational embeddings $X\dashrightarrow V(A)$ (see Proposition 6.3). The interesting thing is that from Theorem 6.4 we obtain rational embeddings of $X$ into the norm hypersurface $V(B)$.

A further variety which can be associated with a central simple algebra, and therefore with the corresponding Brauer--Severi variety $X$, is the symmetric power $S^m(X)$. These varieties are studied by Krashen and Saltman in \cite{KS}. Let $A$ and $B$ be central simple algebras of the same degree and $X$ and $Y$ the corresponding Brauer--Severi varieties. As a simple consequence of the results given in \cite{KS} we obtain that if $A$ and $B$ generate the same subgroup, then there exists always an integer $m<\mathrm{deg}(A)$ such that $S^m(X)$ is birational to $S^m(Y)$ (see Corollary 7.4). 

In conclusion, we can say that the norm hypersurfaces and certain symmetric powers related to central simple algebras $A$ and $B$ are birational, provided $A$ and $B$ generate the same subgroup. The same should be true for (generalized) Brauer--Severi varieties according to the Amitsur conjecture. Moreover, the results of the present paper show that if $A$ and $B$ generate the same subgroup one also has several rational maps, or even rational embeddings, between these varieties. From this point of view, it would be an interesting task to find further varieties related to a central simple algebra $A$ and to study their geometry in terms of the algebraic structure of $A$.\\


{\small \textbf{Conventions}. Throughout this work $k$ denotes an arbitrary ground field if not stated otherwise.

\section{Generalities on central simple algebras}
The main references for Brauer--Severi varieties and central simple algebras are \cite{AR}, \cite{GS} and \cite{SA1}. For the more general notions of Brauer--Severi schemes and Azumaya algebras we refer to \cite{GRO} and \cite{GRO1}.

A \emph{Brauer--Severi variety} of dimension $n$ is a scheme $X$ of finite type over $k$ such that $X\otimes_k L\simeq \mathbb{P}^n$ for a finite field extension $k\subset L$. A field extension $k\subset L$ for which $X\otimes_k L\simeq \mathbb{P}^n$ is called \emph{splitting field} of $X$. Clearly, the algebraic closure $\bar{k}$ is a splitting field for any Brauer--Severi variety. In fact, every Brauer--Severi variety always splits over a finite separable field extension of $k$ (see \cite{GS}, Corollary 5.1.4). By embedding the finite separable splitting field into its Galois closure, a Brauer--Severi variety therefore always splits over a finite Galois extension. It follows from descent theory that $X$ is projective, integral and smooth over $k$. 

Recall, a finite-dimensional $k$-algebra $A$ is called \emph{central simple} if it is an associative $k$-algebra that has no two-sided ideals other than $0$ and $A$ and if its center equals $k$. If the algebra $A$ is a division algebra it is called \emph{central division algebra}. Note that $A$ is a central simple $k$-algebra if and only if there is a finite field extension $k\subset L$, such that $A\otimes_k L \simeq M_n(L)$ (see \cite{GS}, Theorem 2.2.1). This is also equivalent to $A\otimes_k \bar{k}\simeq M_n(\bar{k})$. An extension $k\subset L$ such that $A\otimes_k L\simeq M_n(L)$ is called splitting field for $A$. 

The \emph{degree} of a central simple algebra $A$ is defined to be $\mathrm{deg}(A):=\sqrt{\mathrm{dim}_k A}$. It turns out that the study of central simple $k$-algebras can be reduced to the study of central division algebras. Indeed, according to the \emph{Wedderburn Theorem} (see \cite{GS}, Theorem 2.1.3), for any central simple $k$-algebra $A$ there is an unique integer $n>0$ and a division algebra $D$ such that $A\simeq M_n(D)$. The division algebra $D$ is also central and unique up to isomorphism. 

Now the degree of the unique central division algebra $D$ is called the \emph{index} of $A$ and is denoted by $\mathrm{ind}(A)$. The index of a central simple $k$-algebra $A$ is also the smallest among the degrees of finite separable field extensions that split $A$ (see \cite{GS}, Corollary 4.5.9). 

Moreover, two central simple $k$-algebras $A\simeq M_n(D)$ and $B\simeq M_m(D')$ are called \emph{equivalent} if $D\simeq D'$. Recall that the \emph{Brauer group} $\mathrm{Br}(k)$ of a field $k$ is the group whose elements are equivalence classes of central simple $k$-algebras, with addition given by the tensor product of algebras. It is an abelian group with inverse of a central simple algebra $A$ being $A^{op}$. The neutral element is the equivalence class of $k$. It is a fact that the Brauer group of any field is a torsion group. The order of a central simple $k$-algebra $A\in \mathrm{Br}(k)$ is called the \emph{period} of $A$ and is denoted by $\mathrm{per}(A)$. 

Denoting by $\mathrm{BS}_n(k)$ the set of all isomorphism classes of Brauer--Severi varieties of dimension $n$ and by $\mathrm{CSA}_{n+1}(k)$ the set of all isomorphism classes of central simple $k$-algebras of degree $n+1$, there is a canonical identification
\begin{center}
$\mathrm{CSA}_{n+1}(k)=\mathrm{BS}_n(k)$ 
\end{center}
via non-commutative Galois cohomology (see \cite{AR}, \cite{GS} for details). Hence any $n$-dimensional Brauer--Severi variety $X$ corresponds to a central simple $k$-algebra of degree $n+1$. In view of the one-to-one correspondence between Brauer--Severi varieties and central simple algebras it is also common to speak about the period or index of a Brauer--severi variety $X$, meaning the period or index of the corresponding central simple $k$-algebra. 
We say a Brauer--Severi variety is \emph{minimal} if it corresponds to a central division algebra. If $\mathcal{X}$ is the Brauer--Severi variety corresponding to $A=M_n(D)$ and $X$ the one corresponding to $D$, then $X$ can always be embedded into $\mathcal{X}$ as a closed subvariety (see \cite{GS}, Proposition 5.3.2). Any closed subvariety $X\subset \mathcal{X}$ such that $X\otimes_k L\simeq \mathbb{P}^s$ for some $s\leq \mathrm{dim}(\mathcal{X})$ is called \emph{linear subvariety} and the corresponding central simple algebra represents the same element in $\mathrm{Br}(k)$ as the central simple algebra corresponding to $\mathcal{X}$. The linear subvariety of smallest possible dimension is isomorphic to the minimal Brauer--Severi variety.  

To a central simple algebra $A$ of degree $n$ one can also associate the \emph{generalized} Brauer--Severi variety. It is defined as the projective subvariety of $\mathrm{Grass}_k(nr,A)$ parameterizing the collection of rank $nr$ right ideals of $A$. Here $1\leq r\leq n-1$. If $X$ is the Brauer--Severi variety corresponding to $A$, the generalized Brauer--Severi variety is denoted by $X_r$. By definition $X_1=X$. It can be shown that $X_r$ becomes a Grassmannian after base change to some finte Galois field extension of $k$. For details we refer to \cite{BL}.   
\section{Amitsur conjecture for Brauer--Severi varieties}

Recall the following theorem proved in \cite{AM} (see also \cite{GS}, Theorem 5.4.1):
\begin{thm}
Let $X$ be a Brauer--Severi variety corresponding to the central simple $k$-algebra $A$. Denote by $F(X)$ the function field of $X$. Then the kernel of the restriction map $\mathrm{Br}(k)\rightarrow \mathrm{Br}(F(X))$, $B\mapsto B\otimes_k F(X)$, is a cyclic group generated by $A$.
\end{thm}
This theorem has the following consequence.
\begin{cor}
Let $X$ and $Y$ be Brauer--Severi varieties and $A$ and $B$ the corresponding central simple $k$-algebras. If $X$ and $Y$ are birational, then $A$ and $B$ generate the same subgroup in $\mathrm{Br}(k)$. 
\end{cor}
Amitsur \cite {AM} asked if the other implication of the corollary holds and formulated the following conjecture, referred to as the \emph{Amitsur conjecture} for Brauer--Severi varieties.
\begin{con2}
\textnormal{Let $X$ and $Y$ be Brauer--Severi varieties and $A$ and $B$ the corresponding central simple $k$-algebras. Assume $\mathrm{deg}(A)=\mathrm{deg}(B)$. If $A$ and $B$ generate the same cyclic subgroup of $\mathrm{Br}(k)$, then $X$ and $Y$ are birational.}
\end{con2}
Note that a weaker result is quite easy to prove. Recall, $X$ and $Y$ are called \emph{stably birational} if $X\times_k\mathbb{P}^n$ is birational to $Y\times_k\mathbb{P}^n$. Now if $A$ and $B$ generate the same cyclic subgroup of $\mathrm{Br}(k)$, then $A\otimes_k F(Y)$ and $B\otimes_k F(Y)$ generate the same subgroup in $\mathrm{Br}(F(Y))$. But $B\otimes_k F(Y)$ corresponds to the Brauer--Severi variety $Y\otimes_k F(Y)$ which has a $F(Y)$-rational point coming from the generic point. Hence $Y\otimes_k F(Y)\simeq \mathbb{P}^n$ and therefore the subgroup generated by $B\otimes_k F(Y)$ is trivial. Since $A\otimes_k F(Y)$ generates the same subgroup as $B\otimes_k F(Y)$, it has to be trivial, too. But this implies $X\otimes_k F(Y)\simeq \mathbb{P}^n\otimes_k F(Y)$. In particular, both schemes have the same function field $F(X\times_k Y)$. Thus $X\times_k Y$ is birational to $\mathbb{P}^n\times_k Y$ and by changing the role of $X$ and $Y$ we obtain (see \cite{SA1}, Lemma 13.20.).
\begin{prop}
Let $X$ and $Y$ be two Brauer--Severi varieties of the same dimension and $A$ and $B$ the corresponding central simple $k$-algebras. Then $A$ and $B$ generate the same cyclic subgroup of $\mathrm{Br}(k)$ if and only if $X$ is stably birational to $Y$. 
\end{prop}

\begin{rema}
\textnormal{Note that one implication of Proposition 3.3 also holds for generalized Brauer--Severi varieties. So it is plausible to ask if two generalized Brauer--Severi varieties of the same dimension are birational, provided the corresponding central simple algebras generate the same cyclic subgroup (see \cite{KR1}, Question 1.2)}.
\end{rema}
The Amitsur conjecture seems fascinating as it involves connections between the algebraic structure of $A$ and the geometry of $X$. For details in favor of the Amitsur conjecture we refer to \cite{AM}, \cite{KR}, \cite{KR1}, \cite{RO} and \cite{TR}.

Moreover, Koll\'ar \cite{KO} and Hogadi \cite{HO} considered products of conics (Brauer--Severi varieties of dimension one) respectively products of Brauer--Severi surfaces and proved the following: Let $P_i$ and $Q_j$ be finite collections of conics (resp. Brauer--Severi surfaces) and suppose the subgroup in $\mathrm{Br}(k)$ generated by all $P_i$ equals the subgroup generated by all $Q_j$, then $\prod_i{P_i}$ is birational to $\prod_j{Q_j}$. So it is also plausible and natural to extend the Amitsur conjecture to such products. 
  
\section{Automorphisms of Brauer--Severi varieties}
Below we prove a theorem concerning automorphisms of Brauer--Severi varieties. It will be needed frequently in the next sections.\\ 

Without loss of generality, we assume the field $k$ to be infinite. This is no restriction as for finite fields the Brauer group is trivial and hence $\mathbb{P}^n$ are the only Brauer--Severi varieties. For trivial reasons, we not need to consider this case. We fist prove the following lemma.
\begin{lem}
Let $k\subset E$ be a finite Galois extension. Given a $G_{E/k}:=\mathrm{Gal}(E|k)$ set $P$ of ${n+1}$ distinct points in $\mathbb{P}^n_E$ one can find a set $Q$ of $n+1$ points, where no $n$ points lie in a hyperplane, such that $P$ and $Q$ are isomorphic as $G_{E/k}$ sets.
\end{lem} 
\begin{proof}
We can find $n+1$ distinct points $\{\alpha_1,...,\alpha_{n+1}\}\subset \mathbb{A}^1(E)$ invariant under $G_{E/k}$ action. Now consider the points $\beta_i=(1:\alpha_i:\alpha^2_i:...:\alpha^n_i)$ in $\mathbb{P}^n_E$ and set $Q=\{\beta_1,...,\beta_{n+1}\}$. As all $\alpha_i$ are distinct, the determinant $\prod_{1\leq i<j\leq n+1}{(\alpha_j-\alpha_i)}$ of the Vandermond matrix 
\begin{center}
$
\begin{pmatrix}
1& 1& \dots	 & 1      \\
\alpha_1	&  \alpha_2	& \dots  & \alpha_{n+1} 	  \\
\vdots	& \vdots& \ddots & \vdots \\
\alpha^n_1 	& \alpha^n_2& \dots	 & \alpha^n_{n+1}
\end{pmatrix}^T
$
\end{center} is non-zero and hence no $n$ points of $Q=\{\beta_1,...,\beta_{n+1}\}$ lie in a hyperplane, i.e $Q$ is non-collinear. Note that the $G_{E/k}$-action on the $\beta_i$ is the same as the action on the $\alpha_i$. 
\end{proof}

\begin{thm}
Let $X$ be a Brauer--Severi variety over $k$ of index $d$. Suppose $x_0, x_1\in X$ are closed points with $k$-algebra isomorphism $k(x_0)\simeq k(x_1)$. Suppose furthermore $k\subset k(x_0)$ is a separable extension of degree $d$. Then there is an automorphism $\phi\in\mathrm{Aut}_k(X)$ with $\phi(x_0)=x_1$.
\end{thm}
\begin{proof}
By assumption $L:=k(x_0)$ is a splitting field for $X$. Now denote by $E$ the Galois closure of $L$. As $x_0$ and $x_1$ are defined over $L$, they are also defined over $E$. Since $E$ is a splitting field for $X$, we have an isomorphism $\psi\colon X\otimes_k E\rightarrow \mathbb{P}^{n}_E$. By assumption, the points $x_0$ and $x_1$ split over $E$ as two sets of $d$ distinct points $P=\{\alpha_1,...,\alpha_d\}$ and $Q=\{\beta_1,...,\beta_d\}$ in $\mathbb{P}^n_E$. Note that the points $\alpha_1,...,\alpha_d$ have to be distinct, since $d$ is the smallest among the degrees of separable splitting fields for $X$. The same holds for $Q$. Since $d$ divides $n+1$, it is easy to verify that one can enlarge the sets $P$ and $Q$ to sets of $n+1$ distinct points $P'=\{\alpha_1,...,\alpha_d,\alpha_{d+1},...,\alpha_{n+1}\}$ and $Q'=\{\beta_1,...,\beta_d,\beta_{d+1},...,\beta_{n+1}\}$ in $\mathbb{P}^n_E$, with $P'$ and $Q'$ being $G_{E/k}$-orbits. 

As $\mathrm{Aut}_E(\mathbb{P}^n_E)$ acts transitively on sets of $n+1$ points in general position, we apply Lemma 4.1 to the set $P'$ and obtain a set $R=\{R_1,...,R_{n+1}\}$ of non-colinear points with the same $G_{E/k}$-action as the one on $P'$. So we set $\psi(\alpha_i)=R_i$ for $i=1,2,...,n+1$. In the same way we can construct an isomorphism $\phi:X\otimes_k E\rightarrow \mathbb{P}^n_E$ with $\phi(\beta_i)=R_i$ for $i=1,2,...,n+1$. Now we get two cocycles $\eta_{\sigma}=\psi(^{\sigma}\psi^{-1})$ and $\zeta_{\sigma}=\phi(^{\sigma}\phi^{-1})$ and observe that they are cohomologous in $H^1(G_{E/k},\mathrm{PGL}_n(E))$, since both correspond to the same Brauer--Severi variety. By the construction of $\phi$ and $\psi$, and the fact that all considered points have the same $G_{E/k}$-action, we see that $\eta_{\sigma}$ and $\zeta_{\sigma}$ can be interpreted as cocycles in $Z^1(G_{E/k},T)$, where $T$ denotes the set
\begin{center}
$T=\{M\in\mathrm{PGL}_n(E): A(R_i)=R_i\}$.
\end{center}
We claim that the natural map $i\colon H^1(G_{E/k},T) \rightarrow H^1(G_{E/k},\mathrm{PGL}_n(E))$ is injective. 

To prove the claim, we consider the group $\mathrm{GL}_{n+1}(E)$ and in there the set $U$ of matrices $N$ such that the coordinate vectors of $E^{n+1}$ representing the $R_i$ are the eigenvectors of $N$. We then get the following diagram:
\begin{displaymath}
\begin{xy}
  \xymatrix{
      1 \ar[r]     &   E^* \ar[r]\ar[d]^{id } &  U \ar[r] \ar[d]   &        T\ar[r] \ar[d]      & 1                   \\
      1 \ar[r]             &   E^*\ar[r]               &   \mathrm{GL}_{n+1}(E)\ar[r]   &       \mathrm{PGL}_{n+1}(E)\ar[r] & 1
  }
\end{xy}
\end{displaymath}
Since $U$ is an abelian subgroup, we can look at the long exact sequence of cohomology associated to the above sequence. The part which is of interest is the following:
\begin{displaymath}
\begin{xy}
  \xymatrix{
      H^1(G_{E/k},U) \ar[r]     &   H^1(G_{E/k},T) \ar[r]\ar[d]^{i} &  H^2(G_{E/k},E^*) \ar[d]^{id}                      \\
                  &   H^1(G_{E/k},\mathrm{PGL}_{n+1}(E))\ar[r]               &   H^2(G_{E/k},E^*)
  }
\end{xy}
\end{displaymath}
Now we observe that if $H^1(G_{E/k},U)=0$, the map at the top is injective and hence $i\colon H^1(G_{E/k},T)\rightarrow H^1(G_{E/k},\mathrm{PGL}_{n+1}(E))$. Indeed, write $R=\{R_1,...,R_{n+1}\}$ as the affine $k$-scheme $\mathrm{Spec}(F)$, with $F$ being a suitable $(n+1)$-dimensional $k$-algebra. Now consider the line bundle $\mathcal{O}_{R\otimes_k E}(1)$ over $R\otimes_k E$. One can show that $U$ is the automorphism group of $\mathcal{O}_{R\otimes_k E}(1)$ and is isomorphic to $(F\otimes_k E)^*$. From an extended version of Hilbert's 90 (see \cite{SE}, X.1, Exercise 2) we in fact get $H^1(G_{E/k},(F\otimes_k E)^*)=0$. Therefore $i\colon H^1(G_{E/k},T)\rightarrow H^1(G_{E/k},\mathrm{PGL}_{n+1}(E))$ is injective. Finally, from the injectivity of $i\colon H^1(G_{E/k},T)\rightarrow H^1(G_{E/k},\mathrm{PGL}_n(E))$ we conclude that $\zeta_{\sigma}$ and $\eta_{\sigma}$ are cohomologous via a coboundary with image in $T$, i.e
\begin{center}
$\zeta_{\sigma}=N\eta_{\sigma}(^{\sigma}N^{-1})$,
\end{center}
with $N\in T$. Hence
\begin{center}
$\psi(^{\sigma}\psi^{-1})=N\phi(^{\sigma}\phi^{-1})(^{\sigma}N^{-1})$
\end{center}
so that we finally obtain
\begin{center}
$^{\sigma}(\psi^{-1}N\phi)=\psi^{-1}N\phi$.
\end{center}
Thus $\psi^{-1}N\phi$ descents to an $k$-automorphism of $X$ that maps $x_0$ to $x_1$. This completes the proof.
\end{proof}
\begin{rema}
\textnormal{Note that Theorem 4.2 is proved in \cite{CO} for the special case where $X$ is a Brauer--Severi surface over an algebraic number field.}
\end{rema}
\section{Rational embeddings into Brauer--Severi varieties}
Recall the following definition given in \cite{ME}.
\begin{defi}
\textnormal{Let $X$ and $Y$ be integral, separated schemes of finite type over a field $k$. A \emph{rational embedding} $f\colon X\dashrightarrow Y$ is a rational map such that the restriction to a dense, open subset $U\subset X$ yields an immersion $U\rightarrow Y$. A morphism $U\rightarrow Y$ is an immersion if it gives an isomorphism of $U$ with an open subscheme of a closed subscheme of $Y$.} 
\end{defi}
The following lemma was communicated to me by Daniel Krashen.
\begin{lem}
Let $X$ be a Brauer--Severi variety of index $d$ corresponding to the central simple algebra $M_m(D)$. Let $k\subset L$ be a maximal separable subfield of $D$. Given an arbitrary open subset $U\subset X$, there exists always a closed point $x\in U$ with $k(x)\simeq L$ and $[k(x):k]=d$.
\end{lem}
\begin{proof}
Note that $L\subset D$ is a degree $d$ separable field extension of $k$ such that $X\otimes_k L\simeq \mathbb{P}^n_L$. Now the set of rational points of $X\otimes_k L\simeq \mathbb{P}^n_L$ is a dense subset of $\mathbb{P}^n_L$. In particular, for the open subset $U_L:=U\otimes_k L\subset X\otimes_k L$, the set $U_L(L)=U(L)$ is non-empty. This gives us a morphism $\mathrm{Spec}(L)\rightarrow U$ whose image is a closed point $x\in U$. We then get $k\subset k(x)\subset L$ and, as $x\in U\subset X$ is a closed point, that $k(x)$ is a splitting field for $X$. But since $k\subset L$ is a minimal splitting field, it follows $k(x)\simeq L$. 
\end{proof}
We are now able to prove our first main theorem. 
\begin{thm}
Let $\mathcal{X}$ and $\mathcal{Y}$ be Brauer--Severi varieties corresponding to the central simple $k$-algebras $A=M_n(D)$ and $B=M_n(D')$ with $n>1$ arbitrary and assume $\mathrm{deg}(D)=\mathrm{deg}(D')$. Denote by $X$ and $Y$ the Brauer--Severi varieties corresponding to $D$ and $D'$. Then $A$ and $B$ generate the same cyclic subgroup in $\mathrm{Br}(k)$ if and only if there are rational embeddings $X\dashrightarrow \mathcal{Y}$ and $Y\dashrightarrow \mathcal{X}$.
\end{thm}
\begin{proof}
We first show that if there are rational embeddings $X\dashrightarrow \mathcal{Y}$ and $Y\dashrightarrow \mathcal{X}$, then $A$ and $B$ generate the same cyclic subgroup in $\mathrm{Br}(k)$. 

Consider the rational embedding $f\colon X\dashrightarrow \mathcal{Y}$ and denote by $U$ the domain of $f$ which is an open subset of $X$. By definition, $f$ can be factored as an open followed by a closed immersion 
\begin{displaymath}
\begin{xy}
  \xymatrix{U\ar[r]^{\phi} &Z\ar[r]^{\psi} &\mathcal{Y}}
\end{xy}.
\end{displaymath} Here $Z$ is a suitable closed subscheme of $\mathcal{Y}$. 
Now let $L:=F(X)$ be the function field of $X$. It follows that $L$ is a splitting field for $X$ and we therefore have a $L$-rational point in $U_L(L)=U(L)$. This gives us a morphism $\mathrm{Spec}(L)\rightarrow U$ and exploiting the fact that $\phi$ is an open and $\psi$ a closed immersion, we get a morphism $\mathrm{Spec}(L)\rightarrow \mathcal{Y}$. Thus $\mathcal{Y}(L)\neq\emptyset$ and hence $L$ splits $\mathcal{Y}$. Obviously, $L$ is also a splitting field for $Y$. Repeating the argument for the rational embedding $Y\dashrightarrow \mathcal{X}$ yields that $F(Y)$ splits $X$. So by \cite{RO}, Theorem 5 we conclude that $D$ is Brauer-equivalent to $D'^{\otimes r}$ and $D'$ to $D^{\otimes s}$ for suitable positive integers $r$ and $s$. This implies that $D$ and $D'$, and therefore $A$ and $B$, generate the same cyclic subgroup in $\mathrm{Br}(k)$. 

Now assume that $D$ and $D'$ generate the same cyclic subgroup in $\mathrm{Br}(k)$. First note that by definition $A$ and $B$ also generate the same cyclic subgroup. Since $\mathrm{ind}(A)<\mathrm{deg}(A)$, it follows from \cite{RO}, Theorem 4 that $\mathcal{X}$ and $\mathcal{Y}$ have to be birational. Denote by $\mathcal{U}\subset \mathcal{X}$ and $\mathcal{V}\subset \mathcal{Y}$ the open subsets with $f\colon \mathcal{U} \stackrel{\sim}\rightarrow\mathcal{V}$.

Now let $L\subset D$ be a maximal separable subfield. By Lemma 5.2 we can always find a closed point $x_0\in \mathcal{U}\subset \mathcal{X}$ with $k(x_0)\simeq L$ and $[k(x_0):k]=\mathrm{ind}(A)$. So $x_0$ is defined over the field extension $L=k(x_0)$. Since the field $L$ also splits $X$, there is also a closed point $x_1\in X$ such that $k(x_1)\simeq k(x_0)$. We now can apply Theorem 4.2 to get an $k$-automorphism $\Phi$ of $\mathcal{X}$ mapping $x_1$ to $x_0$. So considering the composition 
\begin{displaymath}
\begin{xy}
  \xymatrix{X\ar[r]^{j} &\mathcal{X}\ar[r]^{\Psi} &\mathcal{X}}
\end{xy},
\end{displaymath}
where $j\colon X\hookrightarrow \mathcal{X}$ is a chosen closed immersion, we see that $\mathcal{U}\cap (\Phi\circ j(X))$ is non-empty. So without loss of generality we can assume that $X$ is embedded into $\mathcal{X}$ in such a way that $\mathcal{U}\cap X$ is non-empty. We set $U:=\mathcal{U}\cap X$. By definition, $U$ is closed in $\mathcal{U}$ and the isomorphism $f\colon \mathcal{U}\stackrel{\sim}\rightarrow\mathcal{V}$ gives us a closed subscheme $f(U)\subset \mathcal{V}$. Therefore, the restriction $f_{|U}\colon U\rightarrow \mathcal{Y}$ is a morphism that can be factored as a closed immersion followed by an open one. As $\mathcal{Y}$ is locally noetherian, this is equivalent to the fact that $f_{|U}\colon U\rightarrow \mathcal{Y}$ can also be factored as an open immersion followed by a closed one. This gives us the existence rational embeddings $X\dashrightarrow \mathcal{Y}$ in the sense of Definition 5.1. 

Repeating the arguments from above for the birational inverse $\mathcal{Y}\dashrightarrow\mathcal{X}$ and a closed embedding $Y\hookrightarrow \mathcal{Y}$ gives us a rational embedding $Y\dashrightarrow \mathcal{X}$. This completes the proof. 
\end{proof}
\begin{rema}
\textnormal{In Theorem 5.3 one only has to require the existence of rational maps $X\dashrightarrow \mathcal{Y}$ and $Y\dashrightarrow \mathcal{X}$ to obtain that $A$ and $B$ generate the same subgroup.}
\end{rema}

Now let $X$ and $Y$ be Brauer--Severi varieties corresponding to division algebras $D$ and $D'$ of the same degree and denote by $\mathcal{Y}$ the Brauer--Severi variety corresponding to $M_n(D')$ for some fixed $n>1$. Suppose there is a rational embedding $X\dashrightarrow \mathcal{Y}$. As mentioned above, this means that there is an open subset $U\subset X$ such that $X\dashrightarrow \mathcal{Y}$ can be factored as $U\rightarrow Z\rightarrow \mathcal{Y}$, where the first arrow is an open and the latter one a closed immersion. In what follows, we want to study the geometric relation between $Z\subset \mathcal{Y}$ and the minimal linear subvariety $Y\subset \mathcal{Y}$. Note that in general the intersection $Z\cap Y$ can be empty. However, the next lemma shows that we can always assume $Z\cap Y$ is non-empty.
\begin{lem}
Let $X\dashrightarrow \mathcal{Y}$ be the rational embedding from above. Then there is an automorphism $\Psi$ of $\mathcal{Y}$ such that $\Psi(Z)\cap Y$ is non-empty.  
\end{lem}
\begin{proof}
Denote by $f$ the composition $U\rightarrow Z\rightarrow \mathcal{Y}$, with $U$ a suitable open subset of $X$. According to Lemma 5.2 there is a closed point $x_0\in U$ with $k(x_0)\simeq L$ , where $L$ is a maximal separable subfield of $D$. By the definition of $f$ we have an isomorphism of $k$-algebras $k(x_0)\simeq k(y_0)$ for $y_0:=f(x_0)\in \mathcal{Y}$. Since $L$ is a splitting field for $Y$, there is a closed point $y_1\in Y$ with $k(y_1)\simeq L$. Considering $Y$ as a linear subvariety of $\mathcal{Y}$ we have $y_1\in \mathcal{Y}$. According to Theorem 4.2 there exists an $k$-automorphism $\Psi$ of $\mathcal{Y}$ mapping $y_0$ to $y_1$. So we obtain the following rational embedding 
\begin{eqnarray}
U\longrightarrow Z\longrightarrow \mathcal{Y}\overset{\Psi}{\longrightarrow}\mathcal{Y}.
\end{eqnarray}
Now consider $U\rightarrow Z \rightarrow \mathcal{Y}\overset{\Psi}{\rightarrow}\mathcal{Y}\supset Y$ and denote by $\mathcal{Z}$ the image $\Psi(Z)$. It is clear from the construction of the rational embedding (1) that $\mathcal{Z}\cap Y$ is non-empty.
\end{proof}
\begin{cor}
Let $g\colon X\dashrightarrow \mathcal{Y}$ be the rational embedding (1) and $Z'$ the irreducible component of $\mathcal{Z}\cap Y$ which is birational to $X$. If $\mathrm{dim}(Z')=\mathrm{dim}(Y)$, then $X$ and $Y$ are birational.
\end{cor} 
\begin{proof}
The assumption $\mathrm{dim}(Z')=\mathrm{dim}(Y)$ implies that $Z'=Y$ and as $X$ is birational to $Z'$ the assertion follows.
\end{proof}
It begs the question of what happens if $\mathrm{dim}(Z')$ is strictly smaller than the dimension of $Y$. Below we will observe that in this case the irreducible components of $\mathcal{Z}\cap Y\subset Y$ are related to the so called \emph{AS-bundles} of $Y$ which were introduced in \cite{N}. We recall the following definition contained in loc. cit..
\begin{defi}
\textnormal{Let $X$ be a $k$-scheme. A locally free sheaf $\mathcal{E}$ of finite rank on $X$ is called \emph{absolutely split} if it splits as a direct sum of invertible sheaves on $X\otimes_k \bar{k}$. For an absolutely split locally free sheaf we shortly write \emph{AS-bundle}.}
\end{defi}
In \cite{N} we classified all $AS$-bundles on proper $k$-schemes. Among others, we studied more closely the $AS$-bundles on Brauer--Severi varieties. In Section 6 of loc. cit. it is proved that the indecomposable $AS$-bundles on a arbitrary Brauer--Severi variety $X$ are locally free sheaves $\mathcal{W}_i$, $i\in \mathbb{Z}$, such that $\mathcal{W}_i\otimes_k \bar{k}\simeq \mathcal{O}(i)^{\oplus \mathrm{ind}(A^{\otimes i})}$. Here $A$ is the central simple algebra corresponding to $X$. These $\mathcal{W}_i$ are unique up to isomorphism and we have $\mathcal{W}^{\vee}_i\simeq \mathcal{W}_{-i}$ and $\mathcal{W}_0\simeq \mathcal{O}_X$. On the level of $K$-theory we observe that these $\mathcal{W}_i$ generate $K_0(X)$. To be precise; for $X\otimes_k\bar{k}\simeq \mathbb{P}^n$ we denote by $h\in K_0(\mathbb{P}^n)$ the class of $\mathcal{O}_{\mathbb{P}^n}(-1)$. We then have the following well-known result (see \cite{Q}, \text{\S}8, Theorem 4.1).
\begin{thm}
The restriction map $\mathrm{res}\colon K_0(X)\rightarrow K_0(\mathbb{P}^n)$ is injective and its image is additively generated by $\mathrm{ind}(A^{\otimes l})\cdot h^l$ with $0\leq l\leq \mathrm{deg}(A)-1$.
\end{thm}
\begin{cor}
Let $X$ be a Brauer--Severi variety corresponding to $A$. Then the $AS$-bundles $\mathcal{W}^{\vee}_i$, $0\leq i\leq \mathrm{deg}(A)-1$, additively generate $K_0(X)$.  
\end{cor}
\begin{proof}
As $\mathcal{W}_i\otimes_k \bar{k}\simeq \mathcal{O}_{\mathbb{P}^n}(i)^{\oplus \mathrm{ind}(A^{\otimes i})}$, we see that $\mathrm{res}([\mathcal{W}^{\vee}_i])=[\mathcal{O}_{\mathbb{P}^n}(i)^{\oplus \mathrm{ind}(A^{\otimes i})}]=\mathrm{ind}(A^{\otimes i})\cdot h^i\in K_0(\mathbb{P}^n)$. The assertion then follows from Theorem 5.8.
\end{proof}

We shortly recall how $K_0(X)$ is related to the Chow groups $CH^i(X)$. 
Note that the Grothendieck group $K_0(X)$ admits a topological filtration
\begin{eqnarray*}
K_0(X)=K_0(X)^{(0)}\supset K_0(X)^{(1)}\supset ...\supset K_0(X)^{(d)}\supset K_0(X)^{(d+1)}=0
\end{eqnarray*} where $d=\mathrm{dim}(X)$ and $K_0(X)^{(i)}$ is the subgroup generated by $[\mathcal{O}_Z]$ as $Z$ runs over all closed subvarieties of codimension $\geq i$ (see \cite{GS}, p.233). We then set $\mathrm{gr}^iK_0(X):=K_0(X)^{(i)}/K_0(X)^{(i+1)}$. 

Now from the Brown--Gersten--Quillen spectral sequence (see \cite{GS}, p.234) we obtain natural surjective maps $\varphi^i_X\colon CH^i(X)\rightarrow \mathrm{gr}^iK_0(X)$. In fact these maps are given by $[Z]\mapsto \mathcal{O}_Z$. If $\mathrm{deg}(A)$ is a prime, the maps $\varphi^i_X$ are isomorphisms (see \cite{GS}, Lemma 8.3.6). 
In this context we have the following theorem (see \cite{KA1}, Theorem 1).
\begin{thm}
Assume $\mathrm{ind}(A)=\mathrm{per}(A)=r$, then for each $0\leq i\leq d$ the map $\mathrm{res}_i\colon \mathrm{gr}^iK_0(X)\rightarrow \mathrm{gr}^iK_0(\mathbb{P}^d)$ is injective and has image generated by $r/(i,r)\cdot(h-1)^i$. Here $(i,r)$ denotes the greatest common divisor of $i$ and $r$.
\end{thm}

Theorem 5.10 now enables us to prove the following result.
\begin{thm}
Let $X\dashrightarrow \mathcal{Y}$ be the rational embedding (1) and assume $\mathrm{per}(D')=\mathrm{ind}(D')=p$ is a prime. If $\mathrm{dim}(\mathcal{Z}\cap Y)<\mathrm{dim}(Y)$, then for all irreducible components $Z'$ of $\mathcal{Z}\cap Y\subset Y$ one has $\varphi^l_Y([Z'])=r_l\cdot((-1)^lp\cdot[\mathcal{O}_Y]+\sum^l_{j=1}{l\choose j}\cdot[\mathcal{W}^{\vee}_j])$ with $l=\mathrm{deg}(D')-i-1$, where $i=\mathrm{dim}(Z')$. 
\end{thm}
\begin{proof}
First note that $l=\mathrm{deg}(D')-i-1$ is the codimension of $Z'$ in $Y$. As $\mathrm{per}(D')=\mathrm{ind}(D')$, we conclude from Theorem 5.10 that the image of the restriction $\mathrm{res}_l\colon \mathrm{gr}^lK_0(Y)\rightarrow \mathrm{gr}^l K_0(\mathbb{P}^{p-1})$ is additively generated by $p/(l,p)\cdot (h-1)^l$. Note that $p/(l,p)=p$ for $0\leq l\leq p-1$ since $p$ is a prime. Now index reduction (see \cite{SA1}, Theorem 5.5) yields $\mathrm{ind}(D'^{\otimes l})=p/(l,p)=p$ for all $0\leq l\leq p-1$. In $K_0(\mathbb{P}^{p-1})$ we then have
\begin{eqnarray*}
p\cdot(h-1)^l&=& p\cdot(\sum^l_{i=0}{l \choose i} h^{l-i}(-1)^i)  \\
 &=& p\cdot{l\choose 0} h^l -p\cdot{l\choose 1} h^{l-1} +...+ p\cdot{l\choose l} h^0(-1)^l\\
&=& {l\choose 0}\mathrm{ind}(D'^{\otimes l})\cdot h^l-{l\choose 1}\mathrm{ind}(D'^{\otimes (l-1)})\cdot h^{l-1}+...+(-1)^l{l\choose l}p\cdot h^0.
\end{eqnarray*}
The proof of Corollary 5.8 shows $\mathrm{res}([\mathcal{W}^{\vee}_j])=[\mathcal{O}_{\mathbb{P}^{p-1}}(-j)^{\oplus \mathrm{ind}(D'^{\otimes j})}]=\mathrm{ind}(D'^{\otimes j})\cdot h^j\in K_0(\mathbb{P}^{p-1})$. As the map $\mathrm{res}\colon K_0(Y)\rightarrow K_0(\mathbb{P}^{p-1})$ preserves the topological filtration, we immediately get
\begin{eqnarray*}
\mathrm{res}^{-1}_l(p\cdot(h-1)^l)=(-1)^lp\cdot[\mathcal{O}_Y]+\sum^l_{j=1}{l\choose j}\cdot[\mathcal{W}^{\vee}_j].
\end{eqnarray*} 
Since $\mathrm{deg}(D')=p$ is a prime, the map $\varphi^l_Y\colon CH^l(Y)\rightarrow \mathrm{gr}^lK_0(Y)$ is an isomorphism. Applying again Theorem 5.10 gives $\varphi^l_Y([Z'])=r_l\cdot((-1)^lp\cdot[\mathcal{O}_Y]+\sum^l_{j=1}{l\choose j}\cdot[\mathcal{W}^{\vee}_j])$ for a suitable $r_l\in \mathbb{Z}$ with $l$ being the codimension of $Z'$ in $Y$. This completes the proof.
\end{proof}
Assume there is a rational embedding $Y\dashrightarrow \mathcal{X}$ with $\mathcal{X}$ being the Brauer--Severi variety corresponding to $M_n(D)$ as above. We then have a factorization $V\rightarrow \tilde{Z}\rightarrow \mathcal{X}$ where $V\subset Y$ is a suitable open subset, $V\rightarrow \tilde{Z}$ is an open and $\tilde{Z}\rightarrow \mathcal{X}$ a closed immersion. Without loss of generality, we can assume that $\tilde{Z}\cap X$ is non-empty (see Lemma 5.4). From Theorem 5.11 it follows that for all irreducible components $\tilde{Z}''$ of $\tilde{Z}$ we have $\varphi^l_X([\tilde{Z}''])=r'_l\cdot((-1)^lp\cdot[\mathcal{O}_Y]+\sum^l_{j=1}{l\choose j}\cdot[\mathcal{V}^{\vee}_j])$ for a suitable $r'_l\in \mathbb{Z}$ with $l$ being the codimension of $Z''$ in $X$. Here $\mathcal{V}_j$ denote the $AS$-bundles on $X$. 
\begin{prob}
\textnormal{Let $p$ be a prime number and $X$ and $Y$ Brauer--Severi varieties corresponding to division algebras $D$ and $D'$, both of degree $p$. Assume $D$ and $D'$ generate the same subgroup in $\mathrm{Br}(k)$. Determine the values of $r_l$ for and relate them to the values of $r'_l$.}
\end{prob}
\begin{exam}
\textnormal{Let $X\dashrightarrow \mathcal{Y}$ be the rational embedding (1) and $Z'$ the irreducible component which is birational to $X$. If $\mathrm{dim}(Z')=\mathrm{dim}(Y)$, then $Z'=Y$ and hence $Z'\in CH^0(Y)$. Hence $\varphi^0_Y(Z')=\varphi^0_Y(Y)=1\cdot [\mathcal{O}_Y]$, so that $r_0=1$.} 
\end{exam}
\begin{exam}
\textnormal{Again let $X\dashrightarrow \mathcal{Y}$ be the rational embedding (1) and assume $\mathrm{dim}(\mathcal{Z}\cap Y)=0$. This means $\mathcal{Z}\cap Y=\{y_1,...,y_m\}$, where $y_i$ are closed points in $Y$. Denoting by $d_i=[k(y_i)\colon k]$, we immediately get $r_{\mathrm{dim}(Y)}=\sum^m_{i=1} d_i$.} 
\end{exam}
An easy observation is the following.
\begin{cor}
Let $D$ and $D'$ be non-trivial quaternion algebras and $X$ and $Y$ the corresponding conics. Denote by $\mathcal{X}$ and $\mathcal{Y}$ the Brauer--Severi varieties corresponding to $M_n(D)$ and $M_n(D')$ for some $n>1$. Assume $D$ and $D'$ generate the same subgroup, then there are rational embeddings $X\dashrightarrow \mathcal{Y}$ and $Y\dashrightarrow \mathcal{X}$ such that $r_0=r'_0=1$.
\end{cor}
\begin{proof}
Since $D$ and $D'$ generate the same subgroup, $X$ and $Y$ are isomorphic. The same holds for $\mathcal{X}$ and $\mathcal{Y}$. We have closed immersions $X\hookrightarrow \mathcal{X}$ and $Y\hookrightarrow \mathcal{Y}$. Putting the linear subvarieties $X$ and $Y$ in standard position (see \cite{AR}) and exploiting the fact $\mathcal{X}\simeq \mathcal{Y}$, we indeed obtain rational embeddings $X\dashrightarrow \mathcal{Y}$ and $Y\dashrightarrow \mathcal{X}$ such that $r_0=r'_0=1$.   
\end{proof}
\section{Rational embeddings into norm hypersurfaces}
In this section we prove our second main theorem. We first recall the definition of norm hypersurfaces and state an important result due to Saltman. For details we refer to \cite{SA}, \cite{SA1} and \cite{ME}.\\

Let $A$ be a central simple $k$-algebra. Then there is a map $n_A\colon A\rightarrow k$, called the \emph{reduced norm} (see \cite{SA}). This map has several properties, for instance if $A$ is split it is the determinant. Now given a commutative $k$-algebra $S$, the reduced norm map extends uniquely to the algebra $A\otimes_k S$ which may no longer be a central simple algebra. Let $m=n^2$ be the $k$-dimension of $A$. We choose a basis $\{e_1,...,e_m\}$ for $A$ and let $S=k[X_1,...,X_m]$. Now consider the map $n_{A\otimes_k S}\colon A\otimes_k S\rightarrow S$. Let $\gamma=X_1e_1+...+X_me_m\in A\otimes_k S$, where the tensor product of the summands is written by multiplication. Then $n_{A\otimes_k S}(\gamma)=:f\in k[X_1,...,X_m]$ is a polynomial.   
\begin{defi}
\textnormal{Let $f\in k[X_1,...,X_m]$ be the polynomial from above. We define the \emph{norm hypersurface} $V(A)\subset \mathbb{A}^m_k$ to be the closed subvariety defined by $f$.}  
\end{defi}
Note that after base change to some splitting field $L$ of $A$ the norm hypersuface $V(A\otimes_k L)=V(A)\otimes_k L$ becomes isomorphic to the determinantal variety of $A\otimes_k L\simeq M_n(L)$ which is a singular variety for $n\geq 3$. We denote the smooth locus of $V(A)$ by $V^+_A$.\\

In \cite{SA} Saltman proved a variant of the Amitsur conjecture. It is the following result.
\begin{thm}
Let $A$ and $B$ be central simple algebras of the same degree. Then $A$ and $B$ generate the same cyclic subgroup in $\mathrm{Br}(k)$ if and only if $V(A)$ is birational to $V(B)$.
\end{thm}
In loc. cit. it is also proved that for a central simple algebra $A$ of degree $n$ the norm hypersurface $V(A)$ is birational to $\mathcal{X}\times \mathbb{P}^{n^2-n}$, where $\mathcal{X}$ is the Brauer--Severi variety corresponding to $A$. Moreover, in \cite{SA1} Saltman constructed rational embeddings $\mathcal{X}\dashrightarrow V(A)$. These results were then refined by Meth in \cite{ME} and the set of rational embeddings enlarged. Now a strategy to prove the Amitsur conjecture is the following: Let $A$ and $B$ be central simple  algebras of the same degree and $\mathcal{X}$ and $\mathcal{Y}$ the corresponding Brauer--Severi varieties. If $A$ and $B$ generate the same cyclic subgroup in $\mathrm{Br}(k)$, then Theorem 6.2 says that $V(A)$ and $V(B)$ are birational. So the aim is to construct rational embeddings $\mathcal{X}\dashrightarrow V(A)$ and $\mathcal{Y}\dashrightarrow V(B)$ such that the birationalmap between the norm hypersurfaces induces a birational map between $\mathcal{X}$ and $\mathcal{Y}$.

\begin{prop}
Let $A$ be a central simple algebra and $\mathcal{X}$ the corresponding Brauer--Severi variety. Denote by $X$ the minimal linear subvariety of $\mathcal{X}$. Then there are rational embeddings $X\dashrightarrow V(A)$ such that the image of the domain lies in $V^+_A$.  
\end{prop}
\begin{proof}
First of all, if $A$ is a central division algebra, the assertion follows from \cite{SA1}, Chapter 13 or \cite{ME}, Theorem 5.0.31. Note that the constructions given in loc. cit. give a plenty of rational embeddings. Now we assume $A=M_n(D)$ with $n>1$. Take a closed immersion $X\hookrightarrow \mathcal{X}$ and a rational embedding $\mathcal{X}\dashrightarrow V(A)$ as constructed in \cite{SA1} or \cite{ME}. By definition, the rational embedding $\mathcal{X}\dashrightarrow V(A)$ can be factored as $\mathcal{U}\rightarrow Z\rightarrow V(A)$ where the first arrow is an open and the latter one a closed immersion. Here $\mathcal{U}\subset \mathcal{X}$ is a suitable open subset. Note that by construction the domain $\mathcal{U}$ is mapped to $V^+_A$ (see \cite{SA1} and \cite{ME} for details). As in the proof of Theorem 5.3, using Lemma 5.2 and Theorem 4.2 we can assume that $U:=X\cap \mathcal{U}$ is non-empty. As $U$ is closed in $\mathcal{U}$ we obtain the composition $U\rightarrow \mathcal{U}\rightarrow Z$, where the first arrow is a closed and the second one an open immersion. As $Z$ is noetherian, we get a composition $U\rightarrow Z'\rightarrow Z$ with first arrow being an open and second arrow being a closed immersion. Composing this with the closed immersion $Z\rightarrow V(A)$, we finally get a rational embedding $X\dashrightarrow V(A)$. By construction, the image of the domain lies in $V^+_A$. 
\end{proof}
We see that it is always possible to get a rational embedding of $X$ into $V(A)$, but it is not obvious why there should be rational embeddings $Y\dashrightarrow V(A)$. Note that $V(A)$ is the norm hypersurface associated to $A$, whereas $Y$ is the minimal linear subvariety of $\mathcal{Y}$ which corresponds to $B$. The next result shows when this is possible. 

\begin{thm}
Let $A$ and $B$ central simple algebras of the same degree and $\mathcal{X}$ and $\mathcal{Y}$ the corresponding Brauer--Severi varieties. Denote by $X$ and $Y$ minimal linear subvarieties. Then $A$ and $B$ generate the same cyclic subgroup in $\mathrm{Br}(k)$ if and only if there are rational embeddings $X\dashrightarrow V(B)$ and $Y\dashrightarrow V(A)$ such that the image of their domains lies in $V^+_B$ respectively $V^+_A$.
\end{thm}
\begin{proof}
We first show that if $A$ and $B$ generate the same subgroup, then there are rational embeddings $X\dashrightarrow V(B)$ and $Y\dashrightarrow V(A)$ with the desired properties. 

So assume $A$ and $B$ generate the same subgroup. 
Then Theorem 5.3 provides us with rational embeddings $X\dashrightarrow \mathcal{Y}$ and $Y\dashrightarrow \mathcal{X}$. So take a rational embedding $X\dashrightarrow \mathcal{Y}$ and a rational embedding $\Psi\colon \mathcal{Y}\dashrightarrow V(B)$ as constructed in \cite{SA1} or \cite{ME}. By definition, the rational embedding $\Psi\colon \mathcal{Y}\dashrightarrow V(B)$ can be factored as 
\begin{displaymath}
\begin{xy}
  \xymatrix{\mathcal{W}\ar[r]^{h} & W\ar[r]^{\phi} & V(B)},
\end{xy}
\end{displaymath} with $h$ being an open, $\phi$ a closed immersion and $\mathcal{W}$ a suitable open subset of $\mathcal{Y}$. Moreover, the construction of $\Psi\colon \mathcal{Y}\dashrightarrow V(B)$ shows that $\Psi$ maps the domain to the smooth locus $V^+_B\subset V(B)$ (see \cite{SA1} for details).   

Now consider the rational embedding $X\dashrightarrow \mathcal{Y}$. This map can be factored as $U\rightarrow Z\rightarrow \mathcal{Y}$ with closed immersion $Z\hookrightarrow \mathcal{Y}$ and $U$ being a suitable open subset of $X$. Let $L\subset D$ be a maximal separable subfield. According to Lemma 5.2 there is a closed point $x_0\in U$ with $k(x_0)\simeq L$. As $U\rightarrow Z$ is an open immersion, the image of $x_0$ under this map is a closed point $x_1\in Z$ with $k(x_1)\simeq L$. Under the chosen closed immersion $Z\hookrightarrow \mathcal{Y}$, the point $x_1$ is mapped to a closed point $x_2\in \mathcal{Y}$ with $k(x_2)\simeq L$. Thus $L$ is a splitting field for $\mathcal{Y}$ and hence we can find a closed point $y\in \mathcal{W}\subset \mathcal{Y}$ with $k(y)\simeq L$. Note that $[L:k]=\mathrm{ind}(D)$. As $D$ and $D'$ generate the same cyclic subgroup we have $\mathrm{ind}(D)=\mathrm{ind}(D')$ (see \cite{GS}, Corollary 4.5.10) and therefore $[L:k]=\mathrm{ind}(D')$. As in the proof of Theorem 5.3, using Lemma 5.2 and Theorem 4.2 enables us to assume that $Z\cap\mathcal{W}=:U'$ is non-empty. Since $U'$ is closed in $\mathcal{W}$ we obtain the composition
\begin{displaymath}
\begin{xy}
  \xymatrix{U'\ar[r]& \mathcal{W}\ar[r]& W}
\end{xy}
\end{displaymath} 
where the first arrow is a closed and the second one an open immersion. Again, since $W$ is neotherian we get that the composition $U'\rightarrow \mathcal{W}\rightarrow W$ can be factored as $U'\rightarrow \mathcal{W}'\rightarrow W$ where the first map is an open and the second map a closed immersion. This gives us a rational embedding $Z\dashrightarrow V(B)$ as the following composition 
\begin{displaymath}
\begin{xy}
  \xymatrix{U'\ar[r]& \mathcal{W}'\ar[r]& W\ar[r] &V(B)}.
\end{xy}.
\end{displaymath} 
And since $Z$ is birational to $X$, we obtain a rational embedding $X\dashrightarrow V(B)$. By construction, the image of the domain lies in $V^+_B$. Applying the above arguments to some chosen rational embeddings $Y\dashrightarrow \mathcal{X}$ and $\mathcal{X}\dashrightarrow V(A)$ provides us with a rational embedding $Y\dashrightarrow V(A)$ such that the image of the domain lies in $V^+_A$.

Now assume we are given rational embeddings $X\dashrightarrow V(B)$ and $Y\dashrightarrow V(A)$ such that the image of the domains lie in the respective smooth loci. Then take the rational embedding $X\dashrightarrow V(B)$ and its factorization $U\rightarrow Z\rightarrow V(B)$ into an open, followed by a closed immersion. Let $L=F(X)$ be the function field of $X$. As $U_L(L)=U(L)\neq \emptyset$ we have a morphism $\mathrm{Spec}(L)\rightarrow U$. So the composition $U\rightarrow Z\rightarrow V(B)$ gives us a morphism $\mathrm{Spec}(L)\rightarrow V(B)$. This means $V(B)(L)\neq \emptyset$. Since we have a birational map $V(B)\dashrightarrow \mathcal{Y}\times \mathbb{P}^{n^2-n}_k$ with $n=\mathrm{deg}(B)$, the base change to $L$ yields a birational map $V(B)_L\dashrightarrow \mathcal{Y}_L\times \mathbb{P}^{n^2-n}_L$. Applying the Lang--Nishimura Theorem gives us a $L$-rational point for $\mathcal{Y}_L\times \mathbb{P}^{n^2-n}_L$. Therefore we have a morphism $\mathrm{Spec}(L)\rightarrow \mathcal{Y}\times \mathbb{P}^{n^2-n}_k$. Now from this morphism we get a morphism $\mathrm{Spec}(L)\rightarrow \mathcal{Y}\times \mathbb{P}^{n^2-n-1}_k\rightarrow \mathcal{Y}$, where the latter map is the projection onto the first factor. Hence $\mathcal{Y}(L)$ is non-empty, meaning that $L$ is a splitting field for $\mathcal{Y}$. It is also possible to apply Lang--Nishimura to the rational map $V(B)\dashrightarrow \mathcal{Y}$. This map is also constructed in \cite{SA1}. Repeating the same argument for the rational embedding $Y\dashrightarrow V(A)$ implies that $F(Y)$ splits $\mathcal{X}$. And again by \cite{RO}, Theorem 5 we find that $A$ and $B$ generate the same cyclic subgroup in $\mathrm{Br}(k)$.
\end{proof}
\begin{rema}
\textnormal{In Theorem 6.4 one only has to require the existence of rational maps $X\dashrightarrow V(B)$ and $Y\dashrightarrow V(A)$ (with image of the domains lying in the respective smooth loci) to conclude that $A$ and $B$ generate the same subgroup.}
\end{rema}

\section{Rational maps between symmetric powers of Brauer--Severi varieties}
In this section we show under what conditions symmetric powers of Brauer--Severi varieties are (stably) birational.\\

Let $\mathcal{X}$ be the Brauer--Severi variety corresponding to $A=M_n(D)$ and $X$ the minimal linear subvariety. The \emph{symmetric power} $S^m(\mathcal{X})$ of $\mathcal{X}$ is defined to be the quotient of the product $\prod^m_{i=1}\mathcal{X}$ by the symmetric group $S_m$, where the action is given by permutation of the coordinates. It is a projective variety over the base field $k$ and is singular unless $\mathcal{X}$ is one-dimensional. In \cite{KS} it is proved the following.
\begin{thm}
Let $A$ be a central simple $k$-algebra of degree $r$ and $\mathcal{X}$ its Brauer--Severi variety. Then $S^r(\mathcal{X})$ is rational over $k$ and for any $l<r$ the symmetric power $S^l(\mathcal{X})$ is birational to $\mathcal{X}_l\times \mathbb{P}^{l(l-1)}$, where $\mathcal{X}_l$ is the $l$-th generalized Brauer--Severi variety. 
\end{thm}  
\begin{proof}
This is \cite{KS}, Theorem 1.4 and 1.5.
\end{proof}
Exploiting Theorem 5.3, we make the following observation.
\begin{prop}
Let $A=M_n(D)$ and $B=M_n(D')$ be central simple algebras with $n>1$ arbitrary and $\mathcal{X}$ and $\mathcal{Y}$ the corresponding Brauer--Severi varieties. Denote by $X$ and $Y$ the minimal linear subvarieties. If $A$ and $B$ generate the same cyclic subgroup in $\mathrm{Br}(k)$, then there are rational maps $X\dashrightarrow S^m(\mathcal{Y})$ and $Y\dashrightarrow S^m(\mathcal{X})$ for all $m>0$.
\end{prop}
\begin{proof}
As $A$ and $B$ generate the same subgroup, Theorem 5.3 provides us with rational maps $X\dashrightarrow \mathcal{Y}$ and $Y\dashrightarrow \mathcal{X}$, i.e with morphisms $U\rightarrow \mathcal{Y}$ and $V\rightarrow \mathcal{X}$ for suitable open subsets $U\subset X$ and $V\subset Y$. Composing these morphisms with the canonical morphisms $\mathcal{Y}\rightarrow S^m(\mathcal{Y})$ and $\mathcal{X}\rightarrow S^m(\mathcal{X})$ yields the assertion.
\end{proof}
\begin{rema}
\textnormal{Note that by definition $S^1(X)=X$ so that Proposition 7.2 can be reformulated as giving rational maps $S^1(X)\dashrightarrow S^m(\mathcal{Y})$ and $S^1(Y)\dashrightarrow S^m(\mathcal{X})$.}
\end{rema}
We believe that the other implication of the above proposition cannot hold. 
Indeed, let $\mathcal{X}, \mathcal{Y}, X$ and $Y$ be as in Proposition 7.2 and assume there are rational maps $X\dashrightarrow S^m(\mathcal{Y})$ and $Y\dashrightarrow S^m(\mathcal{X})$ for some integer $m<\mathrm{deg}(A)$. 
Let $U\subset X$ be the domain of the rational map $X\dashrightarrow S^m(\mathcal{Y})$ and $L$ the function field $F(X)$ of $X$, which corresponds to the generic point. Consider the induced morphism $\mathrm{Spec}(L)\rightarrow U$. Now assume that the image of $L$ under the rational map $X\dashrightarrow S^m(\mathcal{Y})$ lies in the smooth locus. We then obtain a smooth $L$-rational point in $S^m(\mathcal{Y})$. This gives $S^m(\mathcal{Y}_L)(L)\neq \emptyset$. Theorem 7.1 states that $S^m(\mathcal{Y}_L)$ is birational to $(\mathcal{Y}_L)_m\times \mathbb{P}^{m(m-1)}_L$. Now Lang--Nishimura Theorem provides us with a rational point in $(\mathcal{Y}_L)_m\times \mathbb{P}^{m(m-1)}_L$ and therefore $((\mathcal{Y}_L)_m\times \mathbb{P}^{m(m-1)}_L)(L)\neq \emptyset$. This gives us a morphism $\mathrm{Spec}(L)\rightarrow \mathcal{Y}_m\times \mathbb{P}^{m(m-1)}$ and hence $\mathcal{Y}_m(L)\neq \emptyset$. So $L$ is a $\frac{1}{m}$-splitting field for $B$, meaning that $\mathrm{ind}(B\otimes_k L)$ divides $m$. So in general, there is no reason for $L$ to be a splitting field of $B$. But this would exclude the possibility that $A$ and $B$ generate the same subgroup.\\

Theorem 7.1 now has the following consequence.
\begin{prop}
Let $A$ and $B$ be central simple algebras of the same degree corresponding to the Brauer--Severi varieties $X$ and $Y$. Furthermore, let $m$ be a positive integer with $m<\mathrm{deg}(A)$. If $A$ and $B$ generate the same subgroup, then $S^m(X)$ is stably birational to $S^m(Y)$ for all $m$ satisfying $2m< \mathrm{deg}(A)+1$ and birational to $S^m(Y)$ for all $m$ with $2m\geq\mathrm{deg}(A)+1$.
\end{prop}
\begin{proof}
Denote the degree of $A$ by $n$. If $A$ and $B$ generate the same subgroup, $X_m\times \mathbb{P}^{m(n-m)}$ is birational to $Y_m\times \mathbb{P}^{m(n-m)}$ (see \cite{KR1}, p.690). Theorem 7.1 now states that for $m<n$ the symmetric power $S^m(X)$ is birational to $X_m\times \mathbb{P}^{m(m-1)}$. The same holds for $S^m(Y)$. Hence for $2m<n+1$ we find that $S^m(X)\times \mathbb{P}^{m(n-2m+1)}$ is birational to $Y_m\times \mathbb{P}^{m(n-2m+1)}$. For the case $2m\geq n+1$ we notice that $m(m-1)\geq m(n-m)$ and hence $X_m\times \mathbb{P}^{m(m-1)}$ is birational to $Y_m\times \mathbb{P}^{m(m-1)}$. This completes the proof.
\end{proof}   
\begin{cor}
Let $A$ and $B$ be central simple algebras of the same degree corresponding to the Brauer--Severi varieties $X$ and $Y$. If $A$ and $B$ generate the same subgroup, then there is always an integer $m<\mathrm{deg}(A)$ such that $S^m(X)$ is birational to $S^m(Y)$.
\end{cor}
\begin{proof}
Let $n$ be the degree of $A$. Note that we can always take $m$ to be $n-1$ unless $n<3$. Indeed, for $n\geq 3$ we have $2(n-1)\geq n+1$ and the assertion follows from Proposition 7.4. In the case $n< 3$ we only have to consider $n=2$ as $n=1$ is clear. But for $n=2$ the corresponding Brauer--Severi varieties are one-dimensional. For these Brauer--Severi varieties the Amitsur conjecture holds (see \cite{RO}, \cite{TR}) and hence if $A$ and $B$ generate the same subgroup, $X$ and $Y$ are birational. If $X$ is a conic, the only generalized Brauer--Severi variety is the conic itself. So from Theorem 7.1 we obtain that for $n=2$ the symmetric power $S^1(X)=X$ is birational to $S^1(Y)=Y$. This completes the proof.   
\end{proof}

\section{Equivalent assertions}
I this last section we want to give some statements equivalent to those of Theorem 5.3 and 6.4, all of them being well-known. For a central simple algebra $A$ we write $D^b(A)$ for the bounded derived category of finitely generated right $A$-modules. We first recall the definition of a semiorthogonal decomposition and follow here \cite{O}.

Let $D^b(X)$ be the bounded derived category of coherent sheaves on a smooth projective $k$-scheme and $\mathcal{C}$ a triangulated subcategory. The subcategory $\mathcal{C}$ is called \emph{thick} if it is closed under isomorphisms and direct summands. For a subset $A$ of objects of $D^b(X)$ we denote by $\langle A\rangle$ the smallest full thick subcategory of $D^b(X)$ containing the elements of $A$. Recall that a full triangulated subcategory $\mathcal{D}$ of $D^b(X)$ is called \emph{admissible} if the inclusion $\mathcal{D}\hookrightarrow D^b(X)$ has a left and right adjoint functor.
\begin{defi}
\textnormal{Let $X$ be a smooth projective $k$-scheme. A sequence $\mathcal{D}_1,...,\mathcal{D}_n$ of full triangulated subcategories of $D^b(X)$ is called \emph{semiorthogonal} if all $\mathcal{D}_i\subset D^b(X)$ are admissible and $\mathcal{D}_j\subset \mathcal{D}_i^{\perp}=\{\mathcal{F}^{\bullet}\in D^b(X)\mid \mathrm{Hom}(\mathcal{G}^{\bullet},\mathcal{F}^{\bullet})=0$, $\forall$ $ \mathcal{G}^{\bullet}\in\mathcal{D}_i\}$ for $i>j$.}

\textnormal{Such a sequence defines a \emph{semiorthogonal decomposition} of $D^b(X)$ if the smallest full thick subcategory containing all $\mathcal{D}_i$ equals $D^b(X)$.}
\end{defi}

For a semiorthogonal decomposition of $D^b(X)$ we write $D^b(X)=\langle\mathcal{D}_1,...,\mathcal{D}_r\rangle$.

\begin{exam}
\textnormal{Bernardara \cite{BER} proved that a $n$-dimensional Brauer--Severi variety $X$ corresponding to the central simple algebra $A$ always admits a semiorthogonal decomposition given by $D^b(X)=\langle D^b(k), D^b(A),...,D^b(A^{\otimes n})\rangle$.}
\end{exam}

\begin{thm}
Let $\mathcal{X}$ and $\mathcal{Y}$ be $n$-dimensional Brauer--Severi varieties corresponding to central simple algebras $A=M_r(D)$ and $B=M_r(D')$ respectively. Denote by $X$ and $Y$ the Brauer--Severi varieties corresponding to $D$ and $D'$. Then the following are equivalent:
\begin{itemize}
      \item[\bf (i)] $A$ and $B$ generate the same cyclic subgroup in $\mathrm{Br}(k)$.
			\item[\bf (ii)] There are rational embeddings $X\dashrightarrow \mathcal{Y}$ and $Y\dashrightarrow \mathcal{X}$.
			\item[\bf (iii)] There is a bijective map $\phi:\{0,1,...,n-1\}\rightarrow \{0,1,...,n-1\}$ and triangulated equivalences $D^b(A^{\otimes i})\stackrel{\sim}\rightarrow D^b(B^{\otimes \phi(i)})$ coming from Morita equivalence.
			\item[\bf (iv)] $\mathrm{ind}(A\otimes_k L)=\mathrm{ind}(B\otimes_k L)$ for any field extension $k\subset L$.
			\item[\bf (v)] There are dominant rational maps $X\dashrightarrow Y$ and $Y\dashrightarrow X$.
			\item[\bf (vi)] $X$ and $Y$ are stably birational.
			\item[\bf (vii)] $V(A)$ and $V(B)$ are birational.
			\item[\bf (viii)] There are rational embeddings $X\dashrightarrow V(B)$ and $Y\dashrightarrow V(A)$ such that the image of their domains lies in $V^+_B$ respectively $V^+_A$.
\end{itemize}
\end{thm}
\begin{proof}
The equivalence of (i) and (ii) is Theorem 5.3. The content of \cite{KA}, Lemma 7.13 is exactly the equivalence of (i) and (iv). We now show that (i) is equivalent to (iii). For this, we consider the semiorthogonal decompositions (see Example 8.2) 
\begin{eqnarray*}
D^b(X)=\langle D^b(k), D^b(A),...,D^b(A^{\otimes n})\rangle,\\
D^b(Y)=\langle D^b(k), D^b(B),...,D^b(B^{\otimes n})\rangle.
\end{eqnarray*} From \cite{AN}, Theorem 3.1, it easily follows that $A$ and $B$ generate the same cyclic subgroup in $\mathrm{Br}(k)$ if and only if there is a bijective map $\phi:\{0,1,...,n\}\rightarrow \{0,1,...,n\}$ and triangulated equivalences $D^b(A^{\otimes i})\stackrel{\sim}\rightarrow D^b(B^{\otimes \phi(i)})$ coming from Morita equivalence. The equivalence of (i) and (v) follows from \cite{AM}, but can also be found in \cite{KO}, Lemma 16. Finally, the equivalence of (i) and (vi) is Proposition 3.3, whereas the equivalence of (i) and (vii), respectively (viii), follows from Theorem 6.2, respectively 6.4. 
\end{proof}
Any of the statements in Theorem 8.3 are conjectured to be equivalent to the condition that $X$ is birational to $Y$.

{\small MATHEMATISCHES INSTITUT, HEINRICH--HEINE--UNIVERSIT\"AT 40225 D\"USSELDORF, GERMANY}\\
E-mail adress: novakovic@math.uni-duesseldorf.de

\end{document}